\documentclass[10pt]{amsart}

\usepackage{float}
\usepackage{amsmath}
\usepackage{amsfonts}
\usepackage{amssymb}
\usepackage{graphicx}
\usepackage{hyperref}
\usepackage{mathrsfs}
\usepackage{pb-diagram}
\usepackage{epstopdf}
\usepackage{amsmath,amsfonts,amsthm,enumerate,amscd,latexsym,curves}
\usepackage{bbm}
\usepackage[mathscr]{eucal}
\usepackage{epsfig,epsf,xypic,epic}
\usepackage{caption}
\usepackage{subcaption}
\usepackage{tikz}
\usepackage{geometry}
\usetikzlibrary{matrix,arrows,decorations.pathmorphing}
\newtheorem{theorem}{Theorem}[section]

\newtheorem{definition}[theorem]{Definition}

\newtheorem{proposition}[theorem]{Proposition}

\newtheorem{remark}[theorem]{Remark}
\newtheorem{conjecture}[theorem]{Conjecture}

\newcommand{\dbar}{\bar{\partial}}

\newcommand{\pd}[2]{\frac{\partial #1}{\partial #2}}

\newcommand{\inner}[1]{\langle #1\rangle}

\newcommand{\bb}[1]{\mathbb{#1}}
\newcommand{\cu}[1]{\mathcal{#1}}

\begin{document}
	
	\title[Reconstruction of $T_{\bb{P}^2}$ via tropical Lagrangian multi-section]{Reconstruction of $T_{\bb{P}^2}$ via tropical Lagrangian multi-section}
	%\author[K. Chan]{Kwokwai Chan}
	%\address{Department of Mathematics\\ The Chinese University of Hong Kong\\ %Shatin\\ Hong Kong}
	%\email{kwchan@math.cuhk.edu.hk}
	\author[Y.-H. Suen]{Yat-Hin Suen}
	\address{Center for Geometry and Physics\\ Institute for Basic Science (IBS)\\ Pohang 37673\\ Republic of Korea}
	\email{yhsuen@ibs.re.kr}
	\date{\today}
	
	\begin{abstract}
		In this paper, we study the reconstruction problem of the holomorphic tangent bundle $\mathbb{T}_{\mathbb{P}^2}$ of the complex projective plane $\mathbb{P}^2$. We introduce the notion of tropical Lagrangian multi-section and cook up one by tropicalizing the Chern connection associated the Fubini-Study metric. Then we perform the reconstruction of $\mathbb{T}_{\mathbb{P}^2}$ from this tropical Lagrangian multi-section. Walling-crossing phenomenon will occur in the reconstruction process.
	\end{abstract}
	
	\maketitle
	
	\tableofcontents

	\section{Introduction}\label{sec:intro}
	
	Mirror symmetry is a duality between symplectic geometry and complex geometry. The famous SYZ conjecture \cite{SYZ} allows mathematicians to construct mirror pairs and explain homological mirror symmetry \cite{HMS} geometrically via a fiberwise Fourier--Mukai-type transform, which we call the {\em SYZ transform}.
	
	The SYZ transform has been constructed and applied to understand mirror symmetry in the semi-flat case \cite{AP, LYZ, Leung05, sections_line_bundles} and the toric case \cite{Abouzaid06, Abouzaid09, Fang08, FLTZ11b, FLTZ12, Chan09, Chan-Leung10a, Chan-Leung12, Chan12, Chan-Ueda12, CPU13, CPU14, Fang16}. But in all of these works the primary focus was on Lagrangian sections and the mirror holomorphic line bundles the SYZ program produces. For higher rank sheaves, Abouzaid \cite{Abouzaid_Family_Floer} applied the idea of family Floer cohomology to construct sheaves on the rigid analytic mirror of a given compact Lagrangian torus fibration without singular fiber. Recently, applications of the SYZ transform for unbranched Lagrangian multi-sections have been study by Kwokwai Chan and the author in \cite{CS_SYZ_imm_Lag}.

	In this paper, we study the holomorphic tangent bundle $T_{\bb{P}^2}$ of $\bb{P}^2$ in terms of mirror symmetry. Via SYZ construction, the mirror of $\bb{P}^2$ is given by a Landau-Ginzburg model $(Y,W)$ (see Section \ref{sec:P2} for a brief review). It carries a natural fibration $p:Y\to N_{\bb{R}}\cong\bb{R}^2$, which is dual (up to a Legendre transform) to the moment map fibration of $\bb{P}^2$. Since $T_{\bb{P}^2}$ is naturally an object in the derived category $D^b\text{Coh}(\bb{P}^2)$, it is natural to ask what is its mirror Lagrangian in $Y$. Being a rank 2 bundle, the mirror Lagrangian of $T_{\bb{P}^2}$ is expected to be a rank 2 Lagrangian multi-section of $p:Y\to N_{\bb{R}}$ with certain asymptotic conditions. As the base $N_{\bb{R}}$ of the fibration $p$ is simply connected, every unbranched covering map is trivial. But $T_{\bb{P}^2}$ is certainly indecomposable. Hence we are leaded to consider branched Lagrangian multi-section and the SYZ transform defined in \cite{AP, CS_SYZ_imm_Lag, LYZ} cannot be applied directly. To overcome this technicality, we introduce the notion of \emph{tropical Lagrangian multi-section} and reconstruct $T_{\bb{P}^2}$ from this tropical object.
	
	\begin{definition}[=Definition \ref{def:trop_lag}]
		Let $B$ be an $n$-dimensional integral affine manifold without boundary. A \emph{rank $r$ tropical Lagrangian multi-section} is a triple $\bb{L}:=(L,\pi,\varphi)$, where
		\begin{itemize}
			\item [a)] $L$ is a topological manifold.
			\item [b)] $\pi:L\to B$ a covering map of degree $r$ with branch locus $S\subset B$ being a union of locally closed submanifolds of codimension at least 2.
			\item [c)] $\varphi=\{\varphi_U\}$ is a multi-valued function on $L$ such that on any two affine charts $U,V\subset L\backslash\pi^{-1}(S)$ (with respective to the induced affine structure on $L\backslash\pi^{-1}(S)$via $\pi$),
			$$\varphi_U-\varphi_V=\inner{m,x}+b,$$
			for some $m\in\bb{Z}^n$ and $b\in\bb{R}$.
		\end{itemize}
	\end{definition}
	
	Let $\Sigma$ be the fan corresponds to $\bb{P}^2$ and $v_0,v_1,v_2$ be its primitive generators. The tropical Lagrangian multi-section is obtained by ``tropicalizing" the Chern connection associated to the Fubini-Study metric. We will do this in Section \ref{sec:trop_lag} by considering a family of K\"ahler metrics on $\bb{P}^2$, which gives rise to a family of Chern connections, parameterized by a small real number $\hbar>0$. We compute the limit $\hbar\to 0$ of the connections to obtain the ``tropical connection", which can be regarded as a singular connection on the total space $TN_{\bb{R}}$. From the tropical connection, we can cook up six linear functions $\varphi^{\pm}_k$, $k=0,1,2$, for which $\varphi^{\pm}_k$ are defined on $\sigma_k$, the cone generated by the rays $v_i,v_j$, $i,j\neq k$. Let $\sigma_k^{\pm}$ be two copies of $\sigma_k$ and we consider $\varphi^+_k$ (resp. $\varphi_k^-$) as a function defined on $\sigma_k^+$ (resp. $\sigma_k^-$). A piecewise linear function $\varphi$ and an integral affine manifold $L$ can be obtained by gluing $\varphi^{\pm}_k:\sigma_k^{\pm}\to\bb{R}$ together in a continuous manner. By projecting $\sigma_k^{\pm}$ to $\sigma_k$, we obtain a map $\pi:L\to N_{\bb{R}}$. The triple $\bb{L}:=(L,\pi,\varphi)$ forms a rank 2 tropical Lagrangian multi-section. See Section \ref{sec:trop_lag} for the detailed construction.
	
	Section \ref{sec:reconstruction} will be devoted to reconstructing $T_{\bb{P}^2}$ from the tropical Lagrangian multi-section $\bb{L}$. Let $U_k$ be the affine chart corresponds to the cone $\sigma_k$. For each $\sigma_k^{\pm}$, we associate a trivial line bundle $\cu{L}_k^{\pm}=\cu{O}_{U_k}$ and so, we have a trivial rank $2$ bundle on $U_k$ by taking direct sum. As the summands are indexed by the maximal cones of $L$, the gluing of $\bb{L}$ tell us how to glue these trivial rank 2 bundles together on $U_{ij}:=U_i\cap U_j$. Motivated by \cite{CS_SYZ_imm_Lag}, we can write down three naive transition functions $\tau_{10}^{sf},\tau_{21}^{sf},\tau_{02}^{sf}$ by considering the slope differences of $\varphi'$ on different maximal cones. However, this naive gluing is inconsistent due to the affine. In order to obtain a consistent gluing, we modify the naive transition functions by three invertible factors $\Theta_{10},\Theta_{21},\Theta_{02}$. Put $\tau_{j}':=\tau_{ij}^{sf}\Theta_{ij}$. These new transition functions satisfy the cocycle condition and hence define a rank 2 holomorphic vector bundle defined on $\bb{P}^2$, which we call it the \emph{instanton-corrected mirror} of $\bb{L}$. Moreover, we have
	
	\begin{theorem}[=Theorem \ref{thm:correct_mirror}]
		The instanton-corrected mirror of the tropical Lagrangian multi-section $\bb{L}$ is isomorphic to the holomorphic tangent bundle $T_{\bb{P}^2}$ of $\bb{P}^2$.
	\end{theorem}
	
	In the last section, Section \ref{sec:wall_crossing}, we will discuss the relationship between the factors $\{\Theta_{ij}\}$ and the family Floer theory of the (conjecturally exists) mirror Lagrangian of $T_{\bb{P}^2}$. The Fourier modes $m_{ij}\in M$ of $\Theta_{ij}$, determine three cotangent directions of $N_{\bb{R}}$ the origin. As the Lagrangian $\bb{L}$ is tropical, we cannot expect one can determine which fibers of $p$ will bound holomorphic disk with the honest Lagrangian. Nevertheless, $m_{ij}$ should be regarded as the normal directions of \emph{walls}\footnote{A wall is a codimension 1 submanifold $W\subset N_{\bb{R}}$ for which there is a non-trivial holomorphic disk bounded by the Lagrangian and those fibers over $W$.} emitting from the branched point $0\in N_{\bb{R}}$.
	
	In Appendix \ref{sec:caustic}, we will give detailed review of Fukaya's local model on caustic points and his construction of mirror bundle via deformation theory \cite{Fukaya_asymptotic_analysis}, Section 6.4. This gives a symplecto-geometric explanation for the walling-crossing phenomenon in our reconstruction process.

	\subsection*{Acknowledgment}
	The author is grateful to 	
	Byung-Hee An, Kwokwai Chan, Ziming Nikolas Ma and Yong-Geun Oh for useful discussions. A special thanks goes to Katherine Lo for her encouragement during this work. 
	
	\section{SYZ mirror symmetry of $\bb{P}^2$}\label{sec:P2}
	
	We begin with reviewing some elementary facts about the complex projective plane $\bb{P}^2$.
	
	Let $N\cong\bb{Z}^2$ be a lattice of rank $2$ and set
	$$N_{\bb{R}}:=N\otimes_{\bb{Z}}\bb{R},\quad M:=Hom_{\bb{Z}}(N,\bb{Z}),\quad M_{\bb{R}}:=M\otimes_{\bb{Z}}\bb{R}.$$
	Let $\Sigma$ be the fan with primitive generators
	$$
	v_0:=(1,1)
	\quad v_1:=(-1,0),
	\quad v_2:=(0,-1),
	$$
	The associate toric variety $X_{\Sigma}$ is the complex projective plane $\bb{P}^2$. The dense torus of $\bb{P}^2$ can be identified with $TN_{\bb{R}}/N$. Let $\check{p}:TN_{\bb{R}}/N\to N_{\bb{R}}$ be the natural projection. Denote the coordinates on $N_{\bb{R}}$ by $\xi^i$ and the fiber coordinates of $\check{p}$ by $\check{y}^i$. Complex coordinates on $TN_{\bb{R}}/N$ are given by
	$$w^i:=e^{z^i}:=e^{\xi^i+\sqrt{-1}\check{y}^i}.$$
	
	There is an 1-1 correspondence\footnote{We use the convention that if a piecewise linear function $f$ is given by $f(v_i)=a_i$, then the corresponding line bundle is given by $\cu{O}\left(\sum_{i=0}^2a_iD_i\right)$.} between supporting functions on $|\Sigma|=N_{\bb{R}}$ and $(\bb{C}^{\times})^2$-equivariant line bundles on $\bb{P}^2$. Explicitly, the equivariant line bundle $\cu{O}(a_0D_0+a_1D_1+a_2D_2)$ corresponds to the supporting function $\varphi:N_{\bb{R}}\to\bb{R}$, defined by setting $\varphi(v_i):=a_i$.
	
	Let
	$$
	\sigma_0:=\bb{R}_{\geq 0}\inner{v_1,v_2}
	\quad \sigma_1:=\bb{R}_{\geq 0}\inner{v_0,v_2},
	\quad \sigma_2:=\bb{R}_{\geq 0}\inner{v_0,v_1}.
	$$
	and $U_k\cong\bb{C}^2$ be the affine chart corresponds to the cone $\sigma_k$, for $k=0,1,2$. We can trivialize $T_{\bb{P}^2}$ on $U_k$ by
	$$\tau_k:T_{\bb{P}^2}|_{U_k}\ni ([\zeta_0:\zeta_1:\zeta_2],v)\mapsto (w^i_k,w^j_k,v^i_k,v^j_k)\in\bb{C}^2\times\bb{C}^2,$$
	for $i,j,k=0,1,2$ distinct and $i<j$. Here,
	$$w^i_k:=\frac{\zeta_i}{\zeta_k} \text{ and }v=:v^i_k\pd{}{w^i_k}+v^j_k\pd{}{w^j_k}.$$
	The transition functions $\tau_{ij}:=\tau_i\circ\tau_j^{-1}$ are given by
	\begin{align*}
		\tau_{10}=\begin{pmatrix}
			-\frac{1}{(w_0^1)^2} & 0
			\\-\frac{w^2_0}{(w_0^1)^2} & \frac{1}{w_0^1}
		\end{pmatrix},
		\tau_{21}=\begin{pmatrix}
			\frac{1}{w_1^2} & -\frac{w^0_1}{(w^2_1)^2}
			\\0 & -\frac{1}{(w^2_1)^2}
		\end{pmatrix},
		\tau_{02}=\begin{pmatrix}
			-\frac{w^1_2}{(w_2^0)^2} & \frac{1}{w_2^0}
			\\-\frac{1}{(w_2^0)^2} & 0
		\end{pmatrix}.
	\end{align*}
	
	\begin{remark}
		We can relate the coordinates $w^i$ on $TN_{\bb{R}}/N$ and the inhomogeneous coordinates $w_0^i$ by setting $w^i=w_0^i|_{(\bb{C}^{\times})^2}$.
	\end{remark}
	
	It is well-known that $\bb{P}^2$ carries a K\"ahler-Einstein metric, called Fubini-Study metric. It is the Hermitian metric associated to the $(1,1)$-form
	$$\omega_{FS}:=2\sqrt{-1}\partial\dbar\phi(\xi),$$
	where
	$$\phi(\xi):=\frac{1}{2}\log(1+e^{2\xi^1}+e^{2\xi^2}).$$
	Let $\nabla_{FS}$ be the Chern connection associated to the the Fubini-Study metric. With respective to the holomorphic frame $\{\pd{}{w^1},\pd{}{w^2}\}$, it can be written as
	\begin{align*}
		\nabla_{FS}|_{U_0}=d-\frac{1}{1+|w|^2}\left[
		\begin{pmatrix}
			2|w^1|^2 & w^1\bar{w}^2
			\\0 & |w^1|^2
		\end{pmatrix}dz^1+\begin{pmatrix}
			|w^2|^2 & 0
			\\\bar{w}^1w^2 & 2|w^2|^2
		\end{pmatrix}dz^2\right],
	\end{align*}
	where $|w|^2:=|w^1|^2+|w^2|^2$. For the purpose of this paper, we need the formula of $\nabla_{FS}|_{U_0}$ in terms of the frame $\{\pd{}{z^1},\pd{}{z^2}\}$. We have
	$$\nabla_{FS}\left(\pd{}{z^i}\right)=\nabla_{FS}\left(w^i\pd{}{w^i}\right)=dz^i\otimes\pd{}{z^i}+w^i\nabla_{FS}\left(\pd{}{w^i}\right).$$
	Hence
	\begin{align*}
		\nabla_{FS}|_{U_0}=d-\frac{1}{1+|w|^2}\left[
		\begin{pmatrix}
			|w^1|^2-|w^2|^2-1 & |w^2|^2
			\\0 & |w^1|^2
		\end{pmatrix}dz^1+\begin{pmatrix}
			|w^2|^2 & 0
			\\|w^1|^2 & |w^2|^2-|w^1|^2-1
		\end{pmatrix}dz^2\right].
	\end{align*}
	
	\subsection{SYZ mirror symmetry of $\bb{P}^2$}
	
	Now, we jump to SYZ mirror symmetry of $\bb{P}^2$.  The mirror of $\bb{P}^2$ is given by the Landau-Ginzburg model $(Y,W)$, where
	\begin{align*}
		Y:=&T^*N_{\bb{R}}/M
		\\W(z_1,z_2):=&z_1+z_2+\frac{q}{z_1z_2},
	\end{align*}
	and $q>0$ is a positive constant. The complex coordinates $z_j$ are given by
	$$z_i:=e^{x_i+\sqrt{-1}y_i},$$
	where $x_i$ is the affine coordinates on $\mathring{P}$ and $y_i$ are the fiber coordinates. One can equip $Y$ with the standard symplectic structure
	$$\omega_Y:=d\xi^1\wedge dy_1+d\xi^2\wedge dy_2,$$
	and the holomorphic volume form
	$$\Omega_Y:=\frac{dz_1}{z_1}\wedge\frac{dz_2}{z_2}.$$
	
	Let $p:Y\to N_{\bb{R}}$ be the natural projection, which is clearly dual to $\check{p}:TN_{\bb{R}}/N\to N_{\bb{R}}$. The homological mirror symmetry conjecture predicts that Lagrangian branes in $Y$ should be mirror to coherent shaves in $\bb{P}^2$. In \cite{Chan09}, Chan consider Lagrangian sections of $p:Y\to N_{\bb{R}}$ with certain decay conditions at infinity and define its SYZ mirror line bundle on $\bb{P}^2$. Roughly speaking, given such a Lagrangian section $L$, one can associate a line bundle $\check{L}\to\bb{P}^2$ together with a connection $\nabla_{\check{L}}$, so that with respective to a local unitary frame $\check{1}$, one has
	$$\nabla_{\check{L}}:=d-\sqrt{-1}\left(f_1(\xi)d\check{y}^1+f_2(\xi)d\check{y}^2\right).$$
	where $(f_1,f_2)$ are local defining equations of $L$. In fact, every such connection is compatible with the metric $e^{-2F}$, where $F$ is a local potential function of the Lagrangian section $L$. In terms of the local holomorphic frame $e^{-F}\check{1}$, the connection becomes
	$$d-\left(f_1(\xi)dz^1+f_2(\xi)dz^2\right).$$
	As an example, by using the potential function $\phi$ of $\omega_{FS}$, for each $k\in\bb{Z}$, we can define the Lagrangian section
	$$L_k:=\{(\xi,k\cdot d\phi(\xi))\in Y:\xi\in N_{\bb{R}}\}.$$
	The mirror bundle of $L_k$ is the line bundle $\cu{O}(k)$ together with the connection
	$$\nabla_k=d-\left(\frac{k|w^1|^2}{1+|w|^2}dz^1+\frac{k|w^2|^2}{1+|w|^2}dz^2\right).$$
	Note that $k\cdot\phi$ is a smoothing of the supporting function corresponds to $\cu{O}(kD_0)\cong\cu{O}_{\bb{P}^2}(k)$ as
	$$\lim_{t\to\infty}\frac{k}{2}\log_t(1+t^{2\xi^1}+t^{2\xi^2})=\max\{0,k\xi^1,k\xi^2\}.$$
	Thus the differential of supporting functions should be regarded as singular Lagrangian sections of $p:Y\to N_{\bb{R}}$.

	\section{A tropical Lagrangian multi-section associated to $T_{\bb{P}^2}$ and reconstruction}\label{sec:trop_data}
	
	In this section, we introduce the notion of tropical Lagrangian multi-section and construct one by tropicalizing the Chern connection associated to the Fubini-Study metric. We then perform the reconstruction of $T_{\bb{P}^2}$ from the tropical Lagrangian multi-section. The wall-crossing phenomenon will be discussed in the last subsection.
	
	We now introduce the following
	
	\begin{definition}\label{def:trop_lag}
		Let $B$ be an $n$-dimensional integral affine manifold without boundary. A \emph{rank $r$ tropical Lagrangian multi-section} is a triple $\bb{L}:=(L,\pi,\varphi)$, where
		\begin{itemize}
			\item [a)] $L$ is a topological manifold.
			\item [b)] $\pi:L\to B$ a covering map of degree $r$ with branch locus $S\subset B$ and ramification locus $S'\subset L$ being a union of locally closed submanifolds of codimension at least 2.
			\item [c)] $\varphi=\{\varphi_U\}$ is a collection of local continuous functions on $L$ such that on any two affine charts $U,V\subset L\backslash S'$ (with respective to the induced affine structure on $L\backslash S'$ via $\pi$),
			$$\varphi_U-\varphi_V=\inner{m,x}+b,$$
			for some $m\in\bb{Z}^n$ and $b\in\bb{R}$.
		\end{itemize}
	\end{definition}
	
	Definition \ref{def:trop_lag} is a straightforward generalization of the notion of polarization in the famous Gross-Siebert program \cite{Gross_Trop_Geo_MS, GS03, GS11}. We also remark that the domain $L$ can be disconnected in general. But in this paper, $L$ is connected and the multi-valued function $\varphi$ is a single-valued continuous function.

	\subsection{Construction of the tropical Lagrangian multi-section}\label{sec:trop_lag}

	In order to obtain a tropical Lagrangian for $T_{\bb{P}^2}$, we need to ``tropicalize" the Chern connection associated to the Fubini-Study metric.
	
	To do this, we construct a family $\{(X_{\hbar},\omega_{\hbar})\}_{\hbar>0}$ of K\"ahler manifolds. Let 
	$$P:=\{(x_1,x_2)\in M_{\bb{R}}:x_1,x_2\geq 0,x_1+x_2\leq 1\}$$
	be a moment polytope of $\bb{P}^2$ and $\mathring{P}$ be its interior. Let
	$$g_{\hbar}:=g_P+\hbar^{-1}\psi.$$
	where
	\begin{align*}
		g_P(x):=&\frac{1}{2}\left(x_1\log(x_1)+x_2\log(x_2)+(1-x_1-x_2)\log(1-x_1-x_2)\right),
		\\\psi(x):=&x_1^2+x_2^2+x_1x_2-x_1-x_2.
	\end{align*}
	The Legendre dual coordinates are denoted by
	$$\xi^i_{\hbar}:=\pd{g_{\hbar}}{x_i}.$$
	They are related to the original coordinates via
	$$\xi_{\hbar}^i=\xi^i+\hbar^{-1}\pd{\psi}{x_i}.$$
	Put
	$$w_{\hbar}^i:=e^{z^i_{\hbar}}:=e^{\xi^i_{\hbar}+\sqrt{-1}\check{y}^i},\quad i=1,2.$$
	A straightforward calculation shows that
	\begin{align*}
		Hess(g_P+\hbar^{-1}\psi)>&0,
		\\\det(Hess(g_P+\hbar^{-1}\psi))=&\frac{1}{\alpha_{\hbar}(x)x_1x_2(1-x_1-x_2)},
	\end{align*}
	for some smooth function $\alpha_{\hbar}:P\to\bb{R}$ so that if we choose $\hbar>0$ small enough, $\alpha_{\hbar}(x)>0$, for all $x\in\mathring{P}$. These are precisely the compatibility conditions stated in \cite{Ab1, Ab2}, which guarantee the complex coordinates $w_{\hbar}^i$ can be extended to $U_0\subset\bb{P}^2$. Thus, we get a family of complex manifolds $\{X_{\hbar}\}_{\hbar>0}$. Each member of this family can be identified with $\bb{P}^2$ via
	$$\iota_{\hbar}:(w_{\hbar}^1,w_{\hbar}^2)\mapsto [1:w_{\hbar}^1:w_{\hbar}^2].$$
	We define
	$$\omega_{\hbar}:=\iota_{\hbar}^*\omega_{FS}$$
	and the associated connection by $\nabla^{\hbar}$, which is given by
	$$
	\nabla^{\hbar}=d-\frac{1}{1+|w_{\hbar}|^2}\left[
	\begin{pmatrix}
	|w_{\hbar}^1|^2-|w_{\hbar}^2|^2-1 & |w_{\hbar}^2|^2
	\\0 & |w_{\hbar}^1|^2
	\end{pmatrix}dz_{\hbar}^1+\begin{pmatrix}
	|w_{\hbar}^2|^2 & 0
	\\|w_{\hbar}^1|^2 & |w_{\hbar}^2|^2-|w_{\hbar}^1|^2-1
	\end{pmatrix}dz_{\hbar}^2\right],
	$$
	under the frame $\{\pd{}{z_{\hbar}^i}\}$ and the coordinates $\{z_{\hbar}^i\}$. We want to compute
	$$\lim_{\hbar\to 0}\nabla^{\hbar}_{\pd{}{z_{\hbar}^i}}\pd{}{z_{\hbar}^j}=\lim_{\hbar\to 0}\sum_{k=1}^2\Gamma_{ij}^k(\hbar)\pd{}{z_{\hbar}^k}.$$
	First of all, it is clear that
	$$\pd{}{z_{\hbar}^i}=\frac{1}{2}\left(\sum_{j=1}^2Hess(\hbar^{-1}\psi+g_P)^{ij}\pd{}{x^j}-\sqrt{-1}\pd{}{\check{y}^j}\right)\to -\frac{\sqrt{-1}}{2}\pd{}{\check{y}^i}.$$
	at every point as $\hbar\to 0$.	Put $g_i:=\pd{g_P}{x_i}$ and $\psi_i:=\pd{\psi}{x_i}$. Note that 
	\begin{align*}
		\psi_1=&\,2x_1+x_2-1,
		\\\psi_2=&\,x_1+2x_2-1
	\end{align*}
	and $\psi_1-\psi_2=x_1-x_2$. To compute $\displaystyle{\lim_{\hbar\to 0}}\Gamma_{ij}^k(\hbar)$, we decompose $P$ into three pieces
	\begin{align*}
		P_0:=& \,P\cap\{\psi_1\leq 0,\psi_2\leq 0\},
		\\P_1:=& \,P\cap\{\psi_1\geq 0,\psi_1\geq\psi_2\},
		\\ P_2:=& \,P\cap\{\psi_2\geq 0,\psi_2\geq\psi_1\}.
	\end{align*}
	For $x\in\mathring{P}_0$, we have
	\begin{align*}
		\lim_{\hbar\to 0}\frac{1}{1+|w_{\hbar}|^2}=&\lim_{\hbar\to 0}\frac{1}{1+e^{2 g_1}e^{2\hbar^{-1}\psi_1}+e^{2 g_2}e^{2\hbar^{-1}\psi_2}}=1
		\\\lim_{\hbar\to 0}\frac{|w^1_{\hbar}|^2}{1+|w_{\hbar}|^2}=&\lim_{\hbar\to 0}\frac{e^{2 g_1}e^{2\hbar^{-1}\psi_1}}{1+e^{2 g_1}e^{2\hbar^{-1}\psi_1}+e^{2 g_2}e^{2\hbar^{-1}\psi_2}}=0,
		\\\lim_{\hbar\to 0}\frac{|w^2_{\hbar}|^2}{1+|w_{\hbar}|^2}=&\lim_{\hbar\to 0}\frac{e^{2 g_2}e^{2\hbar^{-1}\psi_2}}{1+e^{2 g_1}e^{2\hbar^{-1}\psi_1}+e^{2 g_2}e^{2\hbar^{-1}\psi_2}}=0,
	\end{align*}
	as $\psi_1,\psi_2<0$. For $x\in\mathring{ P}_1$, we have
	\begin{align*}
		\lim_{\hbar\to 0}\frac{1}{1+|w_{\hbar}|^2}=&\lim_{\hbar\to 0}\frac{e^{-2\hbar^{-1}\psi_1}}{e^{-2\hbar^{-1}\psi_1}+e^{2g_1}+e^{2 g_2}e^{2\hbar^{-1}(\psi_2-\psi_1)}}=0,
		\\\lim_{\hbar\to 0}\frac{|w^1_{\hbar}|^2}{1+|w_{\hbar}|^2}=&\lim_{\hbar\to 0}\frac{e^{2 g_1}}{e^{-2\hbar^{-1}\psi_1}+e^{2 g_1}+e^{2 g_2}e^{2\hbar^{-1}(\psi_2-\psi_1)}}=1,
		\\\lim_{\hbar\to 0}\frac{|w^2_{\hbar}|^2}{1+|w_{\hbar}|^2}=&\lim_{\hbar\to 0}\frac{e^{ 2g_2}}{e^{-2\hbar^{-1}\psi_2}+e^{2 g_1}e^{2\hbar^{-1}(\psi_1-\psi_2)}+e^{2 g_2}}=0.
	\end{align*}
	Similarly, we have, for $x\in\mathring{P}_2$,
	\begin{align*}
		\lim_{\hbar\to 0}\frac{1}{1+|w_{\hbar}|^2}=&0
		\\\lim_{\hbar\to 0}\frac{|w^1_{\hbar}|^2}{1+|w_{\hbar}|^2}=&0,
		\\\lim_{\hbar\to 0}\frac{|w^2_{\hbar}|^2}{1+|w_{\hbar}|^2}=&1.
	\end{align*}
	The potential function $\phi$ of the Fubini-Study metric defines an isomorphism $d\phi:N_{\bb{R}}\to\mathring{P}$, which is known as the \emph{Legendre transform}. Since $d\phi$ maps $\mathring{\sigma}_i$ to $\mathring{P}_i$, we have a singular connection $\nabla_{FS}^{trop}$ on the total space $TN_{\bb{R}}$, define by
	\begin{align*}
		\nabla_{FS}^{trop}:=\left\{
		\begin{array}{ll}
			d -\sqrt{-1}\begin{pmatrix}
				-1 & 0\\
				0 & 0
			\end{pmatrix}d\check{y}^1-\sqrt{-1}\begin{pmatrix}
				0 & 0\\
				0 & -1
			\end{pmatrix}d\check{y}^2 & \text{ if }\xi\in\mathring{\sigma}_0,
			\\d-\sqrt{-1}\begin{pmatrix}
				1 & 0
				\\0 & 1
			\end{pmatrix}d\check{y}^1-\sqrt{-1}\begin{pmatrix}
				0 & 0
				\\1 & -1
			\end{pmatrix}d\check{y}^2
			&\text{ if }\xi\in\mathring{\sigma}_1,
			\\ d-\sqrt{-1}\begin{pmatrix}
				-1 & 1
				\\0 & 0
			\end{pmatrix}d\check{y}^1-\sqrt{-1}\begin{pmatrix}
				1 & 0
				\\0 & 1
			\end{pmatrix}d\check{y}^2 &\text{ if }\xi\in\mathring{\sigma}_2,
		\end{array}
		\right.
	\end{align*}
	with respective to the frame $\{\sqrt{-1}\pd{}{\check{y}^i}\}$. we can diagonalize the two non-diagonal matrices by
	\begin{align*}
		\begin{pmatrix}
			1 & 0
			\\-1 & 1
		\end{pmatrix}\begin{pmatrix}
			0 & 0
			\\1 & -1
		\end{pmatrix}\begin{pmatrix}
			1 & 0
			\\1 & 1
		\end{pmatrix}=&\,
		\begin{pmatrix}
			0 & 0
			\\0 & -1
		\end{pmatrix},\\
		\begin{pmatrix}
			1 & -1
			\\0 & 1
		\end{pmatrix}
		\begin{pmatrix}
			-1 & 1
			\\0 & 0
		\end{pmatrix}
		\begin{pmatrix}
			1 & 1
			\\0 & 1
		\end{pmatrix}=&\,
		\begin{pmatrix}
			-1 & 0
			\\0 & 0
		\end{pmatrix}.
	\end{align*}
	These amount to a gauge transform of $\nabla_{FS}^{trop}$. Hence with respective to the new frame,
	\begin{align}\label{eqn:trop_lim_fan}
		\nabla_{FS}^{trop}=\left\{
		\begin{array}{ll}
			d -\sqrt{-1}\begin{pmatrix}
				-1 & 0\\
				0 & 0
			\end{pmatrix}d\check{y}^1-\sqrt{-1}\begin{pmatrix}
				0 & 0\\
				0 & -1
			\end{pmatrix}d\check{y}^2 & \text{ if }\xi\in\mathring{\sigma}_0,
			\\d-\sqrt{-1}\begin{pmatrix}
				1 & 0
				\\0 & 1
			\end{pmatrix}d\check{y}^1-\sqrt{-1}\begin{pmatrix}
				0 & 0
				\\0 & -1
			\end{pmatrix}d\check{y}^2
			&\text{ if }\xi\in\mathring{\sigma}_1,
			\\ d-\sqrt{-1}\begin{pmatrix}
				-1 & 0
				\\0 & 0
			\end{pmatrix}d\check{y}^1-\sqrt{-1}\begin{pmatrix}
				1 & 0
				\\0 & 1
			\end{pmatrix}d\check{y}^2 &\text{ if }\xi\in\mathring{\sigma}_2,
		\end{array}
		\right.
	\end{align}
	
	\begin{remark}
		It is straightforward to show that the above gauge change is the $\hbar\to 0$ limit of the (non-holomorphic) change of frame
		$$H_{\hbar}:\begin{cases}
		\pd{}{z_{\hbar}^1}\mapsto \pd{}{z_{\hbar}^1}-\frac{|w_{\hbar}^1|^2}{1+|w_{\hbar}|^2}\pd{}{z_{\hbar}^2},\\
		\pd{}{z_{\hbar}^2}\mapsto \pd{}{z_{\hbar}^2}-\frac{|w_{\hbar}^2|^2}{1+|w_{\hbar}|^2}\pd{}{z_{\hbar}^1}.
		\end{cases}$$
		One can easily check that $dH_{\hbar}\cdot H_{\hbar}^{-1}\to 0$ as $\hbar\to 0$ and obtain (\ref{eqn:trop_lim_fan}).
	\end{remark}
	
	Now we construct a tropical Lagrangian multi-section from $\nabla_{FS}^{trop}$. For each cone $\sigma_k$, we let $\sigma_k^{\pm}$ be two copies of $\sigma_k$. Define six linear functions $\varphi_k^{\pm}:\sigma_k^{\pm}\to\bb{R}$ by
	\begin{align*}
		\varphi_0^-:&\,(\xi^1,\xi^2)\mapsto -\xi^1,
		\\\varphi_1^+:&\,(\xi^1,\xi^2)\mapsto \xi^1,
		\\\varphi_2^-:&\,(\xi^1,\xi^2)\mapsto \xi^2.
		\\\varphi_0^+:&\,(\xi^1,\xi^2)\mapsto -\xi^2,
		\\\varphi_1^-:&\,(\xi^1,\xi^2)\mapsto \xi^1-\xi^2,
		\\\varphi_2^+:&\,(\xi^1,\xi^2)\mapsto \xi^2-\xi^1
	\end{align*}
	We obtain a topological space $L$ by gluing $\sigma^{\pm}_0$ with $\sigma^{\mp}_1$ and $\sigma^{\mp}_2$ along $v_1$ and $v_2$, respectively, and glue $\sigma^{\pm}_1$ with $\sigma^{\mp}_2$ along $v_0$. The topological space $L$ is homeomorphic to $N_{\bb{R}}$ and by choosing a branch cut, say along $v_0$, it is easy to see that the projection map $\pi:L\to|\Sigma|\cong N_{\bb{R}}$ given by mapping $\sigma_k^{\pm}\to\sigma_k$ can be identified with the square map $z\mapsto z^2$ on $\bb{C}$. Moreover, $\{\varphi^{\pm}_k\}$ glue to a continuous piecewise linear function $\varphi$ on $L$. See Figure \ref{fig:TP2}. If we take the ``trace" of $\varphi$, we obtain a piecewise linear function that defines the line bundle $\cu{O}(3)$, which is of course isomorphic to $\det(T_{\bb{P}^2})$ as a holomorphic line bundle.
	\begin{figure}[H]
		\centering
		\includegraphics[width=120mm]{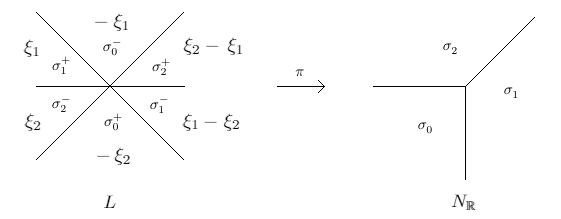}
		\caption{The tropical Lagrangian $\bb{L}$.}
		\label{fig:TP2}
	\end{figure}
	
	In summary, we obtain
	
	\begin{proposition}
		The data $\bb{L}:=(L,\pi,\varphi)$ defines a tropical Lagrangian multi-section.
	\end{proposition}
	
	\begin{remark}
		In \cite{branched_cover_fan}, instead of Hermitian structure, Payne used equivariant structure of $T_{\bb{P}^2}$ to construct the same tropical Lagrangian multi-section.
	\end{remark}
	
	\begin{remark}\label{rmk:trop_Lag}
		One should think of the topological space $L$ and the differential of the function $\varphi$ as the tropical limit of certain rank 2 Lagrangian multi-section of $p:Y\to N_{\bb{R}}$ that is mirror to $T_{\bb{P}^2}$. If we formally apply the SYZ transform to the Lagrangian multi-section $d\varphi:L\to Y$, we obtain (\ref{eqn:trop_lim_fan}). See \cite{CS_SYZ_imm_Lag} for precise definition of SYZ transform of unbranched Lagrangian multi-sections.
	\end{remark}

	\subsection{Reconstructing $T_{\bb{P}^2}$}\label{sec:reconstruction}
	
	As in Remark \ref{rmk:trop_Lag}, the function $\varphi$ should be thought of as the potential function of a Lagrangian multi-section and thus ``$d\varphi$" is the Lagrangian itself. For each $\sigma_k^{\pm}$, let $\cu{L}_k^{\pm}$ be the trivial line bundle $\cu{O}_{U_k}$. Put
	\begin{align*}
		\cu{E}_0:=\cu{L}_0^-\oplus\cu{L}_0^+,\\
		\cu{E}_1:=\cu{L}_1^+\oplus\cu{L}_1^-,\\
		\cu{E}_2:=\cu{L}_2^-\oplus\cu{L}_2^+.
	\end{align*}
	Let $e_k^{\pm}$ be a global frame of $\cu{L}_k^{\pm}$. They give a natural ordered frame for $\cu{E}_k$. According to the gluing of $(L,\varphi)$, we should glue these ordered frame as follows
	\begin{align*}
		(e_0^-,e_0^+)&\leftrightarrow (e_1^+,e_1^-),
		\\(e_1^+,e_1^-)&\leftrightarrow (e_2^-,e_1^+),
		\\(e_2^-,e_2^+)&\leftrightarrow (e_0^+,e_0^-).
	\end{align*}
	The composition of these maps is the monodromy $e_0^{\pm}\mapsto e_0^{\mp}$ of the unbranced covering map $\pi|_{L\backslash\pi^{-1}(0)}:L\backslash\pi^{-1}(0)\to N_{\bb{R}}\backslash\{0\}$. In view of the SYZ transform defined in \cite{CS_SYZ_imm_Lag}, we need to weight each gluing by the exponential of the difference of local potential functions, which is an affine function. Motivated by this, we should look at the monomial associated to the minus of the slope difference in our case. This gives the following set of naive transition functions
	$$\tau^{sf}_{10}:=\begin{pmatrix}
	\frac{a_0}{(w_0^1)^2} & 0
	\\0 & \frac{b_0}{w_0^1}
	\end{pmatrix},
	\tau^{sf}_{21}:=\begin{pmatrix}
	\frac{b_1}{w_1^2} & 0
	\\0 & \frac{a_1}{(w_1^2)^2}
	\end{pmatrix},
	\tau^{sf}_{02}:=\begin{pmatrix}
	0 & \frac{b_2}{w_2^0}
	\\\frac{a_2}{(w_2^0)^2} & 0
	\end{pmatrix},$$
	for some constants $a_i,b_i\in\bb{C}^{\times}$. Since $tr(\varphi)$ is a piecewise linear function defining $\cu{O}(3)$, we should impose the condition
	$$\prod_{i=0}^2a_ib_i=-1.$$
	This condition can be regarded as a rank 1 system on $L\backslash\pi^{-1}(0)$ with monodromy $-1$ as follows. For each maximal cone $\sigma_i^{\pm}\subset L$, choose a neighborhood $V_i^{\pm}\subset L\backslash\pi^{-1}(0)$ of $\sigma_i^{\pm}\backslash\pi^{-1}(0)$ such that $V_i^{\pm}\cap V_j^{\pm}$ is non-empty if and only if $\sigma_i^{\pm}\cap\sigma_j^{\pm}\neq\pi^{-1}(0)$. Hence, by our convention, only $V_i^+\cap V_j^-$ are non-empty, for $i,j$ distinct. Define a rank 1 local system $\cu{L}$ on $L\backslash\pi^{-1}(0)$ by setting its transition functions to be
	\begin{align*}
		a_0 & \text{ on }V_0^-\cap V_1^+,
		\\a_1 & \text{ on }V_1^+\cap V_2^-,
		\\a_2 & \text{ on }V_2^-\cap V_0^+,
		\\b_0 & \text{ on }V_0^+\cap V_1^-,
		\\b_1 & \text{ on }V_1^-\cap V_2^+,
		\\b_2 & \text{ on }V_2^+\cap V_0^-.
	\end{align*}
	The condition $\prod_{i=0}^2a_ib_i=-1$ simply says that $\cu{L}$ has monodromy $-1$. So the constants in the naive transition functions are coupled with the local system data. However, it is clear that $\{\tau_{ij}^{sf}\}$ doesn't satisfy the cocycle condition. This is due to the non-trivial affine monodromy of the branched covering map $\pi:L\to N_{\bb{R}}$.
	
	\begin{remark}
		Although $\{\tau_{ij}^{sf}\}$ does not form a vector vector bundle on $\bb{P}^2$, they do form a rank 2 bundle $\cu{E}^{sf}$ on the singular space $D_0\cup D_1\cup D_2$ because each divisor $D_k$ is covered by the two charts $U_i\cap D_k,U_j\cap D_k$ for $i,j,k=0,1,2$ begin distinct and there are no triple intersections, so the cocycle condition is vacuous. Since the torus $(\bb{C}^{\times})^2\subset\bb{P}^2$ corresponds to the point $0\in N_{\bb{R}}$, which is exactly the branched point of $\pi:L\to N_{\bb{R}}$, from the SYZ transform perspective, $\cu{E}^{sf}$ should be regarded as the mirror bundle of $\bb{L}_0:=(L\backslash\pi^{-1}(0),\pi|_{L\backslash\pi^{-1}(0)},\varphi|_{L\backslash\pi^{-1}(0)})$ and thus deserve the name \emph{semi-flat mirror bundle of $\bb{L}$}.
	\end{remark}
	
	In order to obtain a consistent gluing, we have to modify each $\tau_{ij}^{sf}$ by an invertible factor. We choose the correction factors to be
	\begin{align*}
		\Theta_{10}:=&I+\begin{pmatrix}
			0 & 0
			\\-a_0b_1a_2\frac{w^2_0}{w^1_0} & 0
		\end{pmatrix}\in Aut\left(\cu{E}_0|_{U_{10}}\right),
		\\\Theta_{21}:=&I+\begin{pmatrix}
			0 & -a_0a_1b_2\frac{w^0_1}{w^2_1}
			\\0  & 0
		\end{pmatrix}\in Aut\left(\cu{E}_1|_{U_{21}}\right),
		\\\Theta_{02}:=&I+\begin{pmatrix}
			0 & 0
			\\-b_0a_1a_2\frac{w^1_2}{w^0_2} & 0
		\end{pmatrix}\in Aut\left(\cu{E}_2|_{U_{02}}\right),
	\end{align*}
	For each $j=0,1,2$, the factor $\Theta_{ij}$ is written in terms of the frame on $U_j$ and is only defined on $U_{ij}$. Define $\tau_{ij}':=\tau_{ij}^{sf}\Theta_{ij}$. A straightforward calculation shows that
	
	\begin{proposition}\label{prop:glued} With the condition $\prod_ia_ib_i=-1$, we have
		\begin{equation}\label{eqn:correct_gluing}
			\tau_{02}'\tau_{21}'\tau_{10}'=I.
		\end{equation}
	\end{proposition}
	
	Thus we obtain a rank 2 holomorphic vector bundle $\cu{E}$ on $\bb{P}^2$, which we called the \emph{instanton-corrected mirror} of the tropical Lagrangian multi-section $\bb{L}$. Furthermore, we have
	
	\begin{theorem}\label{thm:correct_mirror}
		The instanton-corrected mirror of the tropical Lagrangian multi-section $\bb{L}$ is isomorphic to the holomorphic tangent bundle $T_{\bb{P}^2}$ of $\bb{P}^2$.
	\end{theorem}
	\begin{proof}
		We define $f:T_{\bb{P}^2}\to\cu{E}$ by
		\begin{align*}
			f|_{U_0}:=f_0:=& \begin{pmatrix}
				1 & 0
				\\0 & a_0b_1a_2
			\end{pmatrix},
			\\f|_{U_1}:=f_1:=& \begin{pmatrix}
				-a_0 & 0
				\\0 & -a_1^{-1}b_2^{-1}
			\end{pmatrix},
			\\f|_{U_2}:=f_2:=& \begin{pmatrix}
				-a_0b_1 & 0
				\\0 & b_2^{-1}
			\end{pmatrix}.
		\end{align*}
		Using $\prod_ia_ib_i=-1$, one can check that
		$$\tau_{02}'f_2=f_0\tau_{02},\quad \tau_{21}' f_1=f_2\tau_{21},\quad \tau_{10}'f_0=f_1\tau_{10}.$$
		Hence $f$ defines an isomorphism.
	\end{proof}

	\begin{remark}\label{rmk:rmk_on_caustic}
		As pointed out by Fukaya \cite{Fukaya_asymptotic_analysis}, Section 6.4, the local system $\cu{L}$ (which he called an orientation twist) is related to the orientation of certain moduli space of holomorphic disks.
	\end{remark}

	\subsection{The wall-crossing factors}\label{sec:wall_crossing}
	
	In Section \ref{sec:reconstruction}, we have introduced three invertible factors
	\begin{align*}
		\Theta_{10}:=&I+\begin{pmatrix}
			0 & 0
			\\-a_0b_1a_2\frac{w^2_0}{w^1_0} & 0
		\end{pmatrix}\in Aut\left(\cu{E}_0|_{U_{10}}\right),
		\\\Theta_{21}:=&I+\begin{pmatrix}
			0 & -a_0a_1b_2\frac{w^0_1}{w^2_1}
			\\0  & 0
		\end{pmatrix}\in Aut\left(\cu{E}_1|_{U_{21}}\right),
		\\\Theta_{02}:=&I+\begin{pmatrix}
			0 & 0
			\\-b_0a_1a_2\frac{w^1_2}{w^0_2} & 0
		\end{pmatrix}\in Aut\left(\cu{E}_2|_{U_{02}}\right),
	\end{align*}
	to modify the naive transition functions $\tau_{ij}^{sf}$. In this section, we give a heuristic explanation about how $\Theta_{ij}$ are related to holomorphic disks bounded by the (conjecturally exists) mirror Lagrangian of $T_{\bb{P}^2}$.
	In terms of the coordinates $w^i=w_0^i$, each factor $\Theta_{ij}$ determines a \emph{Fourier mode}\footnote{By our convention, the Fourier mode of $e^{(m,z)}$ is $-m\in M$}\label{footnote:Fourier} $m_{ij}\in M$, where
	$$m_{10}:=(1,-1),\quad m_{21}:=(0,1),\quad m_{02}:=(-1,0).$$
	Let $n_{ij}\in N$ be the primitive integral tangent vector so that if we identify $N_{\bb{R}},M_{\bb{R}}$ with $\bb{R}^2$ and the natural pairing $\inner{-,-}$ with the standard inner product on $\bb{R}^2$, $\{n_{ij},m_{ij}\}$ forms an orientable orthonormal basis with respective to the standard volume form $dx\wedge dy$ on $\bb{R}^2$. Then we have
	$$n_{10}:=(-1,-1)\quad n_{21}:=(1,0),\quad n_{02}=(0,1),\quad.$$
	See Figure \ref{fig:wall}.
	\begin{figure}[H]
		\centering
		\includegraphics[width=120mm]{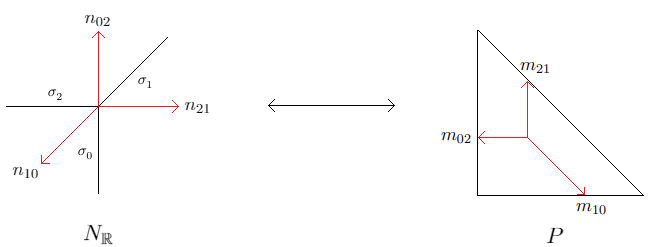}
		\caption{}
		\label{fig:wall}
	\end{figure}
	Hence the cocycle condition (\ref{eqn:correct_gluing}) can be understood as the wall-crossing diagram as shown in Figure \ref{fig:wall}. Furthermore, in view of \cite{HLZ19}, the walls of the mirror Lagrangian should concentrate on a small neighborhood of $\bigcup_{i\neq j}\bb{R}_{\geq 0}\inner{n_{ij}}$ as $\hbar\to 0$.
	
	The closed string wall-crossing phenomenon has been studied in \cite{KS_scattering, GPS_trop_vertex, CLM_scattering}. The wall-crossing factors are elements of the so called \emph{tropical vertex group}. They are responsible for correcting the semi-flat complex structure by Maslov index 0 disks bounded by those SYZ fibers over a wall. In our case, which is an open theory, the factors $\{\Theta_{ij}\}$ are responsible for correcting the ``semi-flat bundle" $\{\tau_{ij}^{sf}\}$ by non-trivial holomorphic disks bounded by SYZ fibers over walls and the mirror Lagrangian of $T_{\bb{P}^2}$. 
	
	We end this section by stating the following
	
	\begin{conjecture}\label{coj:existence}
		There exists a connected rank 2 Lagrangian multi-section $\bb{L}$ of the Lagrangian torus fibration $p:T^*N_{\bb{R}}/N\to N_{\bb{R}}$ so that for any $\epsilon>0$, there exists $\delta>0$ such that if $\hbar\in (0,\delta)$, the $\epsilon$-tubular neighborhood $U_{\epsilon}$ of 
		$$\bigcup_{i\neq j}\bb{R}_{\geq 0}\inner{n_{ij}}$$
		contains the walls of $\bb{L}$. Furthermore, as an object in the Fukaya-Seidel category of $T^*N_{\bb{R}}/N$, $\bb{L}$ is quasi-isomorphic to a cone between the zero section $L_0$ and the direct sum $L_1^{\oplus 3}$.
	\end{conjecture}
	
	As the SYZ transform of $L_0$ and $L_1$ is given by the structural sheaf $\cu{O}$ and the line bundle $\cu{O}(1)$, respectively, this conjecture is nothing but a symplecto-geometric analog of the Euler sequence for $\bb{P}^2$.

	\appendix
	
	\section{Local model for caustics}\label{sec:caustic}

	In this appendix, we give a review on Fukaya's local model on caustic points \cite{Fukaya_asymptotic_analysis}, Section 6.4.
	
	Let $B:=\bb{C}$ and $X:=\bb{C}^2$. Equip $X$ with the standard symplectic structure
	$$\omega=dx^1\wedge dy_1+dx^2\wedge dy_2$$
	and holomorphic volume form
	$$\Omega=dz_1\wedge dz_2,$$
	where $z_i=x^i+\sqrt{-1}y_i$ are the standard complex coordinates on $X$. After a hyperk\"ahler rotation, we have complex coordinates
	$$x=x^1+\sqrt{-1}x^2,\quad y=y_1-\sqrt{-1}y_2.$$
	Fukaya considered the Lagrangian
	$$\overline{L}:=\{(x^2,\bar{x})|x\in B\}\subset X$$
	With respective to $p:X\to B$, the projection onto the first coordinate, $\overline{L}$ is a special Lagrangian multi-section of rank 2. Parameterizing $\overline{L}$ in terms of polar coordinates:
	$$\overline{L}=\{(re^{\sqrt{-1}\theta},\sqrt{r}e^{-\sqrt{-1}\frac{\theta}{2}})\in\bb{C}^2:r\geq 0,\theta\in\bb{R}\}.$$
	Let $u:[0,1]\times[-1,1]\to X$ be given by
	$$u(s,t):=(sre^{\sqrt{-1}\theta},t\sqrt{sr}e^{-\sqrt{-1}\frac{\theta}{2}}).$$
	Define
	$$f_{\overline{L}}(x):=\int_{-1}^1\int_0^1u^*\omega.$$
	An elementary calculation shows that
	$$f_{\overline{L}}(x)=\frac{4}{3}r^{\frac{3}{2}}\cos\left(\frac{3\theta}{2}\right).$$
	
	\begin{proposition}
		There are precisely three gradient flow lines of $f_{\overline{L}}$ starting from the origin.
	\end{proposition}
	\begin{proof}
		In polar coordinates, we have
		$$\nabla f_{\overline{L}}=2\sqrt{r}\cos\left(\frac{3\theta}{2}\right)\pd{}{r}-\frac{2\hbar^{-1}}{\sqrt{r}}\sin\left(\frac{3\theta}{2}\right)\pd{}{\theta}.$$
		Then the gradient flow equation
		$$(\dot{r},\dot{\theta})=\nabla f_{\overline{L}}(r,\theta)$$
		has solution given by
		$$r^{\frac{2}{3}}\sin\left(\frac{3\theta}{2}\right)=C,$$
		where $C$ is a real constant. If the gradient flow lines start from the origin, then we have $C=0$. Hence the gradient flow lines are precisely those straight lines along the directions
		$$\theta=0,\quad\theta=\frac{2\pi}{3},\quad \theta=\frac{4\pi}{3}.$$
	\end{proof}
	
	The three gradient flow lines emitting from the origin of $B$, namely,
	$$(0,\frac{1}{2})\ni t\mapsto te^{\sqrt{-1}\theta},\text{ with }\theta=0,\frac{2\pi}{3},\frac{4\pi}{3},$$
	should correspond to three holomorphic disks bounded by $\overline{L}$ and those fibers of $p:X\to B$ supported on these rays.
	
	Let $\check{p}:\check{X}\to B$ be the dual fibration of $p:X\to B$ and define
	\begin{align*}
		B_{>0}:=&B\backslash\{x\in\bb{C}:x_1\leq 0\},
		\\B_{>0}:=&B\backslash\{x\in\bb{C}:x_1\geq 0\},
		\\\check{X}_{>0}:=&\check{p}^{-1}(B_{>0}),
		\\\check{X}_{<0}:=&\check{p}^{-1}(B_{<0}).
	\end{align*}
	Equip $\overline{L}_0:=\overline{L}\backslash\{0\}$ with a local system $\cu{L}$ with holonomy $-1$. The mirror bundle $\cu{E}_0$ of $(\overline{L}\backslash\{0\},\cu{L})$ has local holomorphic frame $\check{e}_1,\check{e}_2$ on $\check{X}_{>0}$. More precisely,
	$$\check{e}_j=e^{-\frac{2\pi}{\hbar}f_i}\check{1}_i,\quad i=1,2,$$
	The monodromy action around the fiber $\{(0,0)\}\times T^2$ is given by
	\begin{align*}
		\check{e}_1\mapsto&\check{e}_2,
		\\\check{e}_2\mapsto&-\check{e}_1.
	\end{align*}
	
	Let $\dbar_0$ be the Dolbeault operator of $\cu{E}_0$. We need to extend this complex structure to the whole space $X$. Let's delete a small disk $D$ around the origin of $B$. Let $\delta>0$ be small and $\mathfrak{b}_{\delta}$ be a 1-from on $\bb{R}$, supported on $[-\delta,\delta]$ and $\int_{\bb{R}}\mathfrak{b}_{\delta}=1$. Define three elements in $A^{0,1}(\check{X}_{>0},End(\cu{E}_0))$:
	\begin{align*}
		\check{B}_0:=&-Hev(\inner{x,v_0})\cu{F}(\Pi^*_0\mathfrak{b}_{\delta})\check{e}_1^*\otimes\check{e}_2,
		\\\check{B}_1:=&Hev(\inner{x,v_1})\cu{F}(\Pi^*_1\mathfrak{b}_{\delta})\check{e}_2^*\otimes\check{e}_1,
		\\\check{B}_2:=&Hev(\inner{x,v_1})\cu{F}(\Pi^*_2\mathfrak{b}_{\delta})\check{e}_2^*\otimes\check{e}_1,
	\end{align*}
	where $Hev:\bb{R}\to\bb{R}$ is the Heaviside function:
	\begin{align*}
		Hev(x)=\left\{
		\begin{array}{ll}
			1 &\text{ if }x\geq 0,
			\\0 &\text{ if }x<0,
		\end{array}
		\right.
	\end{align*}
	$\Pi_j:\bb{R}^2\to \bb{R}\cdot v_j^{\perp}$, $j=0,1,2$, are the orthogonal projections:
	$$\Pi_j(x):=x-\inner{x,v_j}v_j,$$
	and $\cu{F}$ is the Fourier transform sending $dx^i$ to $d\bar{z}^i$. By choosing $\delta>0$ small enough (such a choice of $\delta$ depends on the radius of the disk $D$), we may assume $\check{B}_0,\check{B}_1,\check{B}_2$ have disjoint support on $\check{p}^{-1}(B_{>0}\backslash D)$. Let
	$$\check{B}:=\check{B}_0+\check{B}_1+\check{B}_2\in A^{0,1}(\check{p}^{-1}(B_{>0}\backslash D),\text{End}(\cu{E}_0)).$$
	Since the support of $\check{B}_i$'s are all away from the ray $\{(x_1,x_2)|x_1\leq 0\}$, $\check{B}$ can be extended to $\check{p}^{-1}(B\backslash D)$. Clearly, $\dbar_0\check{B}=0$ and $[\check{B},\check{B}]=0$ since $\check{B}_j$'s have disjoint support. Therefore, we have
	$$\dbar_0\check{B}+\frac{1}{2}[\check{B},\check{B}]=0,$$
	which means $\dbar_0+\check{B}$ defines a holomorphic structure on the rank 2 complex vector bundle  $\cu{E}_0|_{\check{p}^{-1}(B_{>0}\backslash D)}$.
	
	\begin{proposition}\label{prop:monodromy_free}
		The holomorphic structure $\dbar_0+\check{B}$ is monodromy free around the fiber $\check{p}^{-1}(0)$. Hence $(\cu{E}_0,\dbar_0+\check{B})$ extends to a holomorphic bundle $\cu{E}$ on $\check{X}$.
	\end{proposition}
	
	\bibliographystyle{amsplain}
	\bibliography{geometry}
	
\end{document}